\documentclass[12pt,a4paper]{article}

\usepackage{graphicx}%
\usepackage{amsmath,amsthm,amsfonts,amssymb} 
\usepackage{txfonts,eucal} 
\usepackage{color}

\DeclareMathOperator\Aut{Aut}%
\DeclareMathOperator\Char{Char}%
\DeclareMathOperator\id{id}%
\DeclareMathOperator\End{End}%
\DeclareMathOperator\GL{GL}%
\DeclareMathOperator{\GammaL}{\Gamma L}
\DeclareMathOperator\Kernel{K}%
\DeclareMathOperator\tr{tr}

\newcommand{\li}{\langle}
\newcommand{\re}{\rangle}
\newcommand{\lire}{\li\,\cdot\,,\cdot\,\re}

\newcommand{\ol}[1]{\overline{#1}}
\newcommand{\x}{\times} 

\newcommand{\rip}{\mathrel{\parallel_{r}}}
\newcommand{\lep}{\mathrel{\parallel_{\ell}}}
\newcommand{\ripp}{\mathrel{\parallel_{r}'}}
\newcommand{\lepp}{\mathrel{\parallel_{\ell}'}}
\newcommand{\riS}{\cS_{r}}
\newcommand{\leS}{\cS_{\ell}}
\newcommand{\Autle}{\ensuremath{\Gamma_{\ell}}}
\newcommand{\Autri}{\ensuremath{\Gamma_{r}}}
\newcommand{\Autp}{\ensuremath{\Gamma_\parallel}}
\newcommand{\inner}[1]{\ensuremath{\widetilde{#1}}}

\newcommand{\trenn}{{-\hspace{0pt}}}

\newcommand{\bPH}{\bP(H_F)} 
\newcommand{\dsp}{(\bP,{\lep},{\rip})} 
\newcommand{\dspH}{\bigl(\bPH,{\lep},{\rip}\bigr)}
\newcommand{\cLH}{\cL(H_F)} 
\newcommand{\cAH}{\cA(H_F)} 

\renewcommand{\phi}{\varphi}

\newcommand{\cA}{{\mathcal A}} 

\newcommand{\cF}{{\mathcal F}}

\newcommand{\cL}{{\mathcal L}}
\newcommand{\cM}{{\mathcal M}}

\newcommand{\cO}{{\mathcal O}}

\newcommand{\cS}{{\mathcal S}}

\newcommand{\bP}{{\mathbb P}}

\newtheorem{thm}{Theorem}[section]
\newtheorem{prop}[thm]{Proposition}
\newtheorem{cor}[thm]{Corollary}
\newtheorem{lem}[thm]{Lemma}
\theoremstyle{definition}

\theoremstyle{remark}
\newtheorem{rem}[thm]{Remark}

\renewcommand{\theenumi}{\alph{enumi}}%
\begin{document}

\author{Hans Havlicek \and Stefano Pasotti\thanks{This work was partially
supported by GNSAGA of INdAM (Italy)} \and Silvia
Pianta\footnotemark[1]\,\,\thanks{This work was partially supported by
Universit\`{a} Cattolica del Sacro Cuore (Milano, Italy) in the framework of Call
D.1 2019}}
\title{Characterising Clifford parallelisms among Clifford-like parallelisms}
\date{\today}

\maketitle

\begin{center}
\textit{Dedicated to Mario Marchi on the occasion of his 80th birthday, in
friendship}
\end{center}

\begin{abstract}
We recall the notions of \emph{Clifford} and \emph{Clifford-like} parallelisms
in a $3$-dimensional projective double space. In a previous paper the authors
proved that the linear part of the full automorphism group of a Clifford
parallelism is the same for all Clifford-like parallelisms which can be
associated to it. In this paper, instead, we study the action of such group on
parallel classes thus achieving our main results on characterisation of the
Clifford parallelisms among Clifford-like ones.
\end{abstract}

\section{Introduction}\label{se:intro}

It is a widely used strategy in mathematics to define a new structure by
modifying a given one. The definition of a Clifford-like parallelism from
\cite{blunck+p+p-10a} and \cite{havl+p+p-19a}, which is recalled in
Section~\ref{se:clifford}, follows these lines. The starting point is a
\emph{projective double space} $\dsp$, that is, a projective space $\bP$
together with a \emph{left parallelism} $\lep$ and a \emph{right parallelism}
$\rip$ on its line set such that the so-called \emph{double space axiom} (DS)
is satisfied. The given parallelisms $\lep$ and $\rip$ are called the
\emph{Clifford parallelisms} of $\dsp$ in analogy to the classical example
arising from the three-dimensional elliptic space over the real numbers. The
parallel classes of $\lep$ and $\rip$ are then used to define parallelisms that
are \emph{Clifford-like} w.r.t.\ $\dsp$. Among them are the initially given
parallelisms $\lep$ and $\rip$. We restrict ourselves most of the time to the
case when $\bP$ is three-dimensional, and we make use of an algebraic
description of such a double space in terms of an appropriate four-dimensional
algebra $H$ over a commutative field $F$. Thereby we adopt the notation $\dspH$
and we have to distinguish two cases, \eqref{A} and \eqref{B}. In case
\eqref{A}, $H$ is a quaternion skew field with centre $F$, the left and right
parallelisms do not coincide and, in general, there are Clifford-like
parallelisms of $\dspH$ different from $\lep$ and $\rip$. In case \eqref{B},
$H$ is a commutative extension field of $F$ satisfying some extra property, and
${\lep}={\rip}$ is the only Clifford-like parallelism of $\dspH$. We include
case \eqref{B} for the sake of completeness and in order to obtain a unified
exposition that covers both cases, even though several of our results are
trivial in case \eqref{B}.

\par
In Section~\ref{se:auto+orbits} we study automorphisms of a Clifford-like
parallelism of a projective double space $\dspH$ being motivated by the
following result: if a projective collineation of $\bPH$ preserves \emph{at
least one} Clifford-like parallelism of $\dspH$, then \emph{all} its
Clifford-like parallelisms are preserved.\footnote{The situation gets intricate
when dealing with a non-projective collineation that preserves at least one
Clifford-like parallelism of $\dspH$. See the examples in
\cite[Sect.~4]{havl+p+p-20a}.} This follows from \cite[Thm.~3.5]{havl+p+p-20a}
in case \eqref{A} and holds trivially in case \eqref{B}. In our algebraic
setting these projective collineations are induced by $F$-linear
transformations of $H$ which are described in Subsection~\ref{subse:auto},
where we determine all $F$-semilinear automorphisms of the right parallelism.
In preparation for Section~\ref{se:main}, we exhibit for a quaternion skew
field $H$ the orbits of certain points and lines of $\bPH$ under the group of
inner automorphisms of $H$ and we determine all $\rip$-classes that are fixed
under a left translation of $H$.
\par
The main results are stated in Section~\ref{se:main}. In
Theorem~\ref{thm:double-double}, we consider a three-dimensional projective
space $\bP$ that is made into a double space in two ways. If there exists a
parallelism $\parallel$ on $\bP$ that is Clifford-like w.r.t.\ both double
space structures then the given double spaces coincide up to a change of the
attributes ``left'' and ``right'' in one of them. This finding improves
\cite[Thm.~4.15]{havl+p+p-19a} (see Corollary~\ref{cor:proper}) and it
simplifies matters considerably. Indeed, when dealing with a Clifford-like
parallelism, there is only one corresponding double space structure in the
background. In Theorems~\ref{thm:main}, \ref{thm:new1} and \ref{thm:new2} we
characterise the Clifford parallelisms among the Clifford-like parallelism of
$\dspH$ via the existence of automorphisms with specific properties. For
example, Theorem~\ref{thm:main} establishes that a Clifford-like parallelism of
$\dspH$ is Clifford precisely when it admits an automorphism that fixes all its
parallel classes and acts non-trivially on the point set of the projective
space $\bPH$.
\par
Next, let us emphasise that some of our investigations are in continuity with
classical results on \emph{dilatations} in \emph{kinematic spaces}. For
example, in our proof of Theorem~\ref{thm:main} we could use the fact that the
existence of a \emph{proper} non-trivial dilatation (namely a non-identical
collineation with a fixed point and the property that all parallel classes
remain invariant) is possible only in the commutative case, \emph{i.e.} in our
case \eqref{B} (see \cite[Teorema~2]{marchi+p-80b} or
\cite[(II.10)]{karz+m-84a}). We decided instead to include a short direct proof
in order to keep the paper self-contained. There are also neat connections to
the theory of \emph{Sperner spaces} and \emph{(generalised) translation
structures}; we refer the interested reader to \cite{bader+l-11a},
\cite{seier-71a}, \cite{seier-73a} and the many references given there.
\par
Finally, another remark seems appropriate. Any Clifford-like parallelism on the
three-dimensional real projective space is Clifford (see
Remark~\ref{rem:infinite}). The Clifford parallelisms on this space are the
only topological parallelisms that admit an automorphism group of dimension at
least $4$; see \cite{loew-19c} and the intimately related articles
\cite{bett+l-17a}, \cite{bett+r-14a}, \cite{loew-17y}, \cite{loew-17z}. In
contrast to our considerations, in this beautiful result only the ``size'' of
an automorphism group is taken into account and not its action on the parallel
classes.

\section{Preliminaries on Clifford and Clifford-like parallelisms}\label{se:clifford}

A \emph{parallelism} $\parallel$ on a projective space $\bP$ is an equivalence
relation on the set $\cL$ of lines such that each point of $\bP$ is incident
with precisely one line from each equivalence class. (If $\bP$ is a finite
projective space then a parallelism is also called a \emph{packing} or a
\emph{resolution}.) For each line $M\in\cL$ we write $\cS(M)$ for the
\emph{parallel class} of $M$, that is, the equivalence class containing $M$.
This notation arises quite naturally, since any parallel class is in fact a
\emph{spread} (of lines) of $\bP$. When considering several parallelisms, we
distinguish among the above notions and symbols by adding appropriate
attributes, subscripts or superscripts. We refer to \cite{betta+t+z-19a},
\cite[Ch.~17]{hirsch-85a}, \cite{john-03a}, \cite{john-10a} and
\cite[\S~14]{karz+k-88} for a wealth of results about parallelisms and further
references.
\par
Let $\bP$ and $\bP'$ be projective spaces with parallelisms $\parallel$ and
$\parallel'$, respectively and let $\kappa$ be a collineation of $\bP$ to
$\bP'$ such that, for all lines $M,N\in \cL$, $M\parallel N$ implies
$\kappa(M)\parallel' \kappa(N)$. Then $\kappa$ takes any $\parallel$-class to a
$\parallel'$-class by \cite[Lemma~2.1]{havl+p+p-20a}. Such a $\kappa$ is
frequently called an \emph{isomorphism}\footnote{A slightly different
terminology will be used when dealing with projective spaces over vector
spaces; see the first paragraph of Section~\ref{subse:auto}.} of
$(\bP,\parallel)$ to $(\bP',\parallel')$.
\par
Suppose that a projective space $\bP$ is endowed with two (not necessarily
distinct) parallelisms, a \emph{left} parallelism $\lep$ and a \emph{right}
parallelism $\rip$. Following \cite{kks-73}, $\dsp$ constitutes a
\emph{projective double space} if the following axiom is satisfied.

\begin{enumerate}
  \item[(DS)] For all triangles $p_0,p_1,p_2$ in $\bP$ there exists a
      common point of the lines $M_1$ and $M_2$ that are defined as
      follows. $M_1$ is the line through $p_2$ that is left parallel to the
      join of $p_0$ and $p_1$, $M_2$ is the line through $p_1$ that is
      right parallel to the join of $p_0$ and $p_2$.
\end{enumerate}
Given a projective double space $\dsp$ each of $\lep$ and $\rip$ is referred to
as a \emph{Clifford parallelism}\footnote{This definition does not include
Clifford parallelisms that arise from octonions (see \cite{blunck+k+s+s-18a},
\cite{vanb-68b}, \cite{vanb-68c}, \cite{vanb-68a}, \cite{vane-29a}). The
(generalised) Clifford parallelisms appearing in \cite[Kap.~12]{gier-82} and
\cite{tyrr+s-71a} are not fully covered.} of $\dsp$. More generally, a
\emph{Clifford-like parallelism} of $\dsp$ is defined as a parallelism
$\parallel$ on $\bP$ such that, for all $M,N\in \cL$, $M\parallel N$ implies
$M\lep N$ or $M\rip N$ (see \cite[Def.~3.2]{havl+p+p-19a}). Each parallel class
of a Clifford-like parallelism $\parallel$ of $\dsp$ is a left or a right
parallel class: see \cite[Thm.~3.1]{havl+p+p-19a}, where this topic appears in
the wider context of ``blends'' of parallelisms. A Clifford-like parallelism of
$\dsp$ is said to be \emph{proper} if it does not coincide with one of $\lep$
and $\rip$. In what follows, whenever we say that a parallelism $\parallel$ on
a projective space $\bP$ is Clifford (respectively Clifford-like) it is
intended that $\bP$ can be made into a double space $\dsp$ such that
$\parallel$ is one of its Clifford (respectively Clifford-like) parallelisms.
\par

An algebraic description---up to isomorphism---of \emph{all} projective double
spaces $\dsp$ that contain at least two distinct lines and satisfy the
so-called ``prism axiom'' was given in \cite{kks-73}. It is based on quaternion
skew fields and purely inseparable commutative field extensions of
characteristic two. According to \cite[Satz~1]{kks-74} and
\cite[Satz~2]{kroll-75}, the prism axiom appearing in \cite{kks-73} is
redundant; see also the surveys in \cite[\S~14]{karz+k-88} and
\cite[pp.\,112--115]{john-03a}. This is why we omit to consider this axiom
here. From now on we exhibit exclusively three-dimensional projective double
spaces.\footnote{In any other dimension (DS) implies ${\lep}={\rip}$, whence
proper Clifford-like parallelisms of $\dsp$ do not exist.} We therefore recall
only their algebraic description in the next few paragraphs.
\par
We adopt the following settings throughout this article: $F$ denotes a
commutative field and $H$ is an $F$-algebra with unit $1_H$ satisfying one of
the following conditions.
\begin{enumerate}
\renewcommand{\theenumi}{\Alph{enumi}}
\item\label{A} $H$ is a quaternion skew field with centre $F1_H$.
\item\label{B} $H$ is a commutative field with degree $[H \mathbin{:}
    F1_H]=4$ and such that $h^2\in F1_H$ for all $h\in H$.
\end{enumerate}
In what follows, we identify any $f\in F$ with $f1_H\in H$, whence $F$ turns
into a subfield of $H$. If $E$ is a subfield of $H$, then $H$ is a left vector
space and a right vector space over $E$. We denote these spaces as ${}_E H$ and
$H_E$, respectively. Whenever $E$ is contained in the centre of $H$, we do not
distinguish between $_E H$ and $H_E$. In each of the cases \eqref{A} and
\eqref{B}, $H_F$ is an infinite \emph{kinematic} (or, in a different
terminology: \emph{quadratic}) $F${\trenn}algebra, \emph{i.e.},
\begin{equation}\label{eq:kinematic}
    h^2\in F + Fh \mbox{~~for all~~} h\in H .
\end{equation}
If \eqref{B} applies then the characteristic $\Char F$ equals two and $H$ is a
purely inseparable extension of $F$.
\par
All $F$-linear endomorphisms of $H_F$ constitute the $F$-algebra $\End(H_F)$.
The \emph{left regular representation} $\lambda\colon H\to \End(H_F)$ sends
each $h\in H$ to the mapping $\lambda(h)=:\lambda_h$ given as
$\lambda_h(x):=hx$ for all $x\in H$. The image $\lambda(H)$ is an isomorphic
copy of the field $H$ within $\End(H_F)$. The elements of the multiplicative
group\footnote{We abbreviate $H\setminus\{0\}$ as $H^*$ and use the same kind
of notation for any field.} $\lambda(H^*)=\GL(H_H)$ are the \emph{left
translations}. Similarly, the \emph{right regular representation} $\rho\colon
H\to \End(H_F)$ sends each $h\in H$ to $\rho(h)=:\rho_h$ given as $\rho_h(x):=
xh$ for all $x\in H$. In this way we obtain $\rho(H)$ as an antiisomorphic copy
of $H$ within $\End(H_F)$ and the group of \emph{right
translations}\footnote{Observe that the zero endomorphism $\lambda_0=\rho_0$ is
\emph{not} among the left and right translations.} $\rho(H^*)=\GL({}_HH)$. For
all $g,h\in H$, the mappings $\lambda_g$ and $\rho_h$ commute. The
multiplicative group $H^*$ admits the representation $\tilde{(\;\,)}\colon
H^*\to \GL(H_F)$ sending each $h\in H^*$ to
$\tilde{h}:=\lambda_h^{-1}\circ\rho_h$, which is an inner automorphism of the
field $H$. Clearly, in case \eqref{B} the group $\inner{H^*}$ comprises only
the identity $\id_H$.
\par
The \emph{projective space} on the vector space $H_F$, in symbols $\bPH$, is
understood to be the set of all subspaces of $H_F$ with \emph{incidence} being
symmetrised inclusion. We adopt the usual geometric terms: \emph{Points},
\emph{lines}, and \emph{planes} of $\bPH$ are the subspaces of $H_F$ with
vector dimension one, two, and three, respectively; the set of all lines is
written as $\cL(H_F)$. The following notions rely on $H_F$ being an
$F$-algebra. In $\bPH$, lines $M$ and $N$ are defined to be \emph{left
parallel}, $M\lep N$, if $\lambda_c(M)=N$ for some $c\in H^*$. Similarly, $M$
and $N$ are said to be \emph{right parallel}, $M\rip N$, if $\rho_c(M)=N$ for
some $c\in H^*$. Then $\dspH$ is a \emph{projective double space}. The
parallelisms $\lep$ and $\rip$ are distinct in case~\eqref{A} and identical in
case~\eqref{B}.

\begin{rem}\label{rem:algebras}
The left and right parallelism w.r.t.\ $(H,+,\cdot)$ are the same as the right
and left parallelism defined by the opposite field of $H$. So, from a geometric
point of view, the choice of the attributes ``left'' and ``right'' is
immaterial.
\par
The multiplication on the field $(H,+,\cdot)$ may be altered without changing
the associated projective double space $\dspH$. Let us choose any $e\in H^*$.
Then we can define a multiplication $\cdot^e$ on $H$ via $ x\cdot^e y := x\cdot
e^{-1}\cdot y$ for all $x,y\in H$. This makes $(H,+,\cdot^e)$ into an
$F$-algebra, which will briefly be written as $H^e$. The left translation
$\lambda_e$ (w.r.t.\ $H$) is an $F$-linear isomorphism of $H$ to $H^e$, whence
the arbitrarily chosen element $e\in H^*$ turns out to be the unit element of
$H^e$. The projective double spaces arising from the $F$-algebras $H$ and $H^e$
are the same, since $\lambda_{h}=\lambda^e_{h\cdot e}$ and
$\rho_{h}=\rho^e_{e\cdot h}$ for all $h\in H^*$.

Let us briefly sketch a more conceptual verification of our second observation.
The point $Fe$ and the parallelisms $\lep$ and $\rip$ can be used to make the
point set $\bPH$ into a two-sided incidence group with unit element $Fe$
\cite[\S 3]{kks-73}. (The prism axiom appearing in \cite{kks-73} can be avoided
\cite[Satz~1]{kks-74}, \cite[Satz~2]{kroll-75}.) Then, using the group
structure on $\bPH$, the $F$-vector space $H$ can be endowed with a
multiplication making it into a field with unit element $e$ (see
\cite[Satz~1]{elle+k-63a} and \cite[Hauptsatz]{waeh-67a}). This field, which
coincides with our $H^e$, therefore provides an alternative description of the
projective double space $\dspH$.
\end{rem}

\begin{rem}
There are various other ways to define a \emph{Clifford parallelism} on a
three-dimensional (necessarily pappian) projective space. We refer to
\cite{bett+r-12a}, \cite{blunck+p+p-10a}, \cite[p.~46]{hav95},
\cite[Sect.~2]{hav97}, \cite{havl-15}, \cite{havl-16a} and the references given
there. On that account, it is our aim to make use only of the above algebraic
approach.
\end{rem}

\par
Let $\cAH\subset\cLH$ denote the star of lines with centre $F1$. By
\eqref{eq:kinematic}, each line $L\in\cAH$ is readily seen to be a maximal
commutative subfield of $H$ and hence an $F$-subalgebra. Next, we recall an
explicit construction that gives \emph{all} Clifford-like parallelisms of
$\dspH$. Upon choosing any $\inner{H^*}$-invariant subset $\cF\subseteq\cAH$,
one obtains a partition of $\cL(H_F)$ by taking the left parallel classes of
all lines in $\cal F$ and the right parallel classes of all lines in
$\cAH\setminus\cF$. This partition determines an equivalence relation, which
turns out to be a Clifford-like parallelism $\parallel$ of $\dspH$. See
\cite[Thm.~4.10]{havl+p+p-19a} for a proof in the case when \eqref{A} applies;
in case \eqref{B} the result is trivial due to ${\parallel}={\lep}={\rip}$.

\begin{rem}\label{rem:kernel}
Let $\parallel$ be any parallelism on $\bPH$ and let $\cS(M)$, $M\in\cL(H_F)$,
be one of its parallel classes. We recall that the \emph{kernel} of the spread
$\cS(M)$ consists of all endomorphisms $\phi$ of the abelian group $(H,+)$ such
that $\phi(N)\subseteq N$ for all $N\in\cS(M)$. This kernel, which will be
denoted by $\Kernel\bigl(H,\cS(M)\bigr)$, is a field; see, for example,
\cite[Thm.~1.6]{luen-80a}. Consequently, if
$\phi\in\Kernel\bigl(H,\cS(M)\bigr)$ and $\phi\neq0$, then $\phi(N)=N$ for all
$N\in\cS(M)$. The following simple reasoning will repeatedly be used. If
$\phi_1,\phi_2\in \Kernel\bigl(H,\cS(M)\bigr)$ satisfy $\phi_1(g)=\phi_2(g)$
for some $g\in H^*$, then $(\phi_1-\phi_2)(g)=0$ forces that $\phi_1-\phi_2$ is
not injective. Therefore $\phi_1-\phi_2$ is the zero endomorphism or, in other
words, $\phi_1=\phi_2$.
\end{rem}

\begin{prop}\label{prop:two}
If Clifford parallelisms\/ $\parallel$ and\/ $\parallel'$ on a
three-dimensional projective space have two distinct parallel classes in
common, then these parallelisms coincide.
\end{prop}
\begin{proof}
By virtue of the algebraic description of all projective double spaces and by
Remark~\ref{rem:algebras}, we may assume the following. The parallelism
$\parallel$ is the right parallelism $\rip$ coming from an $F$-algebra
$(H,+,\cdot)$ subject to \eqref{A} or \eqref{B}. There is a multiplication
${\cdot'}\colon H\times H\to H$ making the $F$-vector space $H_F$ into an
$F$-algebra $(H,+,\cdot')$ subject to \eqref{A} or \eqref{B} such that
$\parallel'$ coincides with the right parallelism $\ripp$ arising from
$(H,+,\cdot')$. These algebras share a common unit element $1\in H^*$, say.
\par
By our assumption, there are distinct lines $L_1,L_2\in\cAH$ such that
$\riS(L_1)=\riS'(L_1)$ and $\riS(L_2)=\riS'(L_2)$. Choose any $z\in L_{n}$
where $n\in\{1,2\}$. Then $\lambda_z$ and $\lambda_z'$ are both in
$\Kernel\bigl(H,\riS(L_{n})\bigr)$. According to Remark~\ref{rem:kernel},
$\lambda_z(1)=z=\lambda_z'(1)$ implies $\lambda_z=\lambda_z'$. Hence
\begin{equation}\label{eq:product}
    z\cdot x = \lambda_z(x)=\lambda_z'(x) = z \cdot' x
    \mbox{~~for all~~} x\in H \mbox{~~and all~~} z\in L_1\cup L_2 .
\end{equation}
\par

More generally, the equality in \eqref{eq:product} is fulfilled for all $x\in
H$ and all $z$ from the subfield of $(H,+,\cdot)$ that is generated by $L_1\cup
L_2$. This subfield coincides with $(H,+,\cdot)$, since $L_1$ is a maximal
subfield of $(H,+,\cdot)$. All things considered, we obtain
$(H,+,\cdot)=(H,+,\cdot')$ and therefore ${\parallel} = {\rip} = {\ripp} =
{\parallel'}$.
\end{proof}

\begin{rem}
Note that the above theorem may alternatively be established by using the
one-to-one correspondence between Clifford parallelisms and external planes to
the Klein quadric (see \cite[Cor.~4.5]{havl-16a}).
\end{rem}

\section{Automorphisms, their orbits and actions}\label{se:auto+orbits}

This section is devoted to deepen the study of the automorphisms of the
Clifford parallelisms of a three-dimensional projective double space {$\dspH$}
as described in Section~\ref{se:clifford}. In particular we obtain a
description of the orbits of certain points and lines under the action of the
group $\inner{H^*}$, and we characterise the right parallel classes fixed (as a
set) by a given left translation. In order to avoid trivialities, we shall
repeatedly confine ourselves to case \eqref{A}. These findings will lead us in
Section~\ref{se:main} to the proof of our main results.

\subsection{Automorphisms}\label{subse:auto}

In this subsection $H$ always denotes an $F$-algebra subject to \eqref{A}
or \eqref{B}.
Given any parallelism $\parallel$ on $\bPH$, we are going to use from now on
the phrase \emph{automorphism of $\parallel$} for any $\beta$ in the general
semilinear group $\GammaL(H_F)$ that acts as a $\parallel$-preserving
collineation on $\bPH$. The symbol $\Gamma_\parallel$ denotes the
\emph{automorphism group} of $\parallel$. This terminology is in accordance
with the one in \cite{havl+p+p-20a}.h

\par
The Clifford parallelisms of the projective double space $\dspH$ give rise to
automorphism groups $\Gamma_{\lep}=:\Autle$ and $\Gamma_{\rip}=:\Autri$. These
groups coincide, that is,
\begin{equation}\label{eq:le=ri}
    \Autle=\Autri .
\end{equation}
In case \eqref{A}, a proof can be derived from \cite[p.~166]{pian-87b}; see
\cite[Sect.~2]{havl+p+p-20a} for further details. In case \eqref{B}, equation
\eqref{eq:le=ri} is trivial. The group $\lambda(H^*)$ of left translations, the
group $\rho(H^*)$ of right translations and the group $\inner{H^*}$ of inner
automorphisms are subgroups of $\Autle=\Autri$.

\begin{lem}\label{lem:kernel}
Let $\riS(M)$ be the right parallel class of a line $M\in\cL(H_F)$. The
elements of the kernel\/ $\Kernel\bigl(H,\riS(M)\bigr)$ are precisely the
mappings $\lambda_g$ with $g$ ranging in the line that contains the point $F1$
and is right parallel to $M$. Consequently,
\begin{equation}\label{eq:lambda(H)}
    \lambda(H)  = \bigcup_{L\,\in\,\cAH} \Kernel\bigl(H,\riS(L)\bigr)
                = \bigcup_{M\,\in\,\cLH} \Kernel\bigl(H,\riS(M)\bigr) .
\end{equation}
A similar result holds with the role of ``left'' and ``right'' interchanged.
\end{lem}
\begin{proof}
There is a $d\in H^*$ such that $F1\subseteq \rho_d(M)=Md$. Choose any $g\in
Md$. Then, for all $h\in H^*$, $\lambda_g(Mdh)=g(Mdh)=(gMd)h\subseteq Mdh$,
whence $\lambda_g\in\Kernel\bigl(H,\riS(M)\bigr)$. Conversely, let
$\phi\in\Kernel\bigl(H,\riS(M)\bigr)$. Then $\phi(1)\in Md$ gives
$\lambda_{\phi(1)}\in \Kernel\bigl(H,\riS(M)\bigr)$, and
$\phi(1)=\lambda_{\phi(1)}(1)$ implies $\phi=\lambda_{\phi(1)}$ according to
Remark~\ref{rem:kernel}.
\par
Equation \eqref{eq:lambda(H)} is now immediate, since each element of $H$ is
contained in at least one line of the star $\cAH$ and each right parallel class
contains a line passing through $F1$.
\end{proof}
Any line $L\in\cAH$ is a commutative quadratic extension field of $F$ contained
in $H$. The above Lemma illustrates the rather obvious result that the
restriction to $L$ of the representation $\lambda$ (respectively $\rho$)
provides an isomorphism of the field $L$ onto the kernel of the right
(respectively left) parallel class of the line $L$. This proves anew that all
left and right parallel classes are regular spreads (see
\cite[4.8~Cor.]{blunck+p+p-10a}, \cite[Prop.~3.5]{havl-15} or
\cite[Prop.~4.3]{havl-16a}). Maybe less obvious is the following conclusion.
Any semilinear transformation $\phi\in\GammaL(H_F)$ that fixes all lines of one
right (respectively left) parallel class is a left (respectively right)
translation and therefore in the automorphism group $\Autle=\Autri$.
\par
In the next proposition we describe the automorphism group $\Autle = \Autri$.
Alternative proofs, which cover only the case when $H$ is a quaternion skew
field, can be retrieved from \cite[Sect.~4]{blunck+k+s+s-18a},
\cite[Thm.~1]{pian-87b} and \cite[Prop.~4.1 and 4.2]{pz90-coll}. Below, we
follow the exposition in \cite[Sect.~2]{havl+p+p-20a}.

\begin{prop}\label{prop:auto}
Let\/ $\dspH$ be a projective double space, where $H$ is an $F$-algebra subject
to\/ \eqref{A} or\/ \eqref{B}. The automorphism group of the right parallelism
satisfies
\begin{equation}\label{eq:auto}
    \Autri = \lambda(H^*)\rtimes \Aut(H/F),
\end{equation}
where\/ $\Aut(H/F)$ denotes the group of all automorphisms of the field $H$
that fix $F$ as a set.
\end{prop}
\begin{proof}
A direct verification shows that the group $\Aut(H/F)$ is a
subgroup of $\Autri$. As we noted at the beginning of this subsection,
the same applies for the group $\lambda(H^*)$. For all $\gamma\in\Autri$ and
all lines $M\in\cLH$, we have
\begin{equation}\label{eq:kernel-image}
    \Kernel\bigl(H,\riS(\gamma(M))\bigr)
    =\gamma\circ \Kernel\bigl(H,\riS(M)\bigr)\circ\gamma^{-1} .
\end{equation}
Using \eqref{eq:lambda(H)}, this implies that $\lambda(H^*)$ is a normal
subgroup of $\Autri$.
\par
Let us choose any $\beta\in\Autri$. We define
$\phi:=\lambda_{\beta(1)}^{-1}\circ\beta$, whence $\phi\in\Autri$ fixes $1\in
H$. In order to verify
\begin{equation}\label{eq:lambda-image}
    \phi\circ\lambda_z\circ\phi^{-1} = \lambda_{\phi(z)}
    \mbox{~~for all~~} z\in H,
\end{equation}
we proceed as follows. There is a line $L$ with $1,z\in L$. Applying
\eqref{eq:kernel-image} to $\gamma:=\phi$ and $M:=L$ gives that
$\phi\circ\lambda_z\circ\phi^{-1}$ as well as $\lambda_{\phi(z)}$ belongs to
$\Kernel\bigl(H,\riS(\phi(L))\bigr)$. Now
$(\phi\circ\lambda_z\circ\phi^{-1})(1)=\lambda_{\phi(z)}(1)$ together with
Remark~\ref{rem:kernel} establishes \eqref{eq:lambda-image}. For all $x,y\in
H$, we have
\begin{equation*}
\begin{aligned}
    \phi(xy)    &=  (\phi\circ\lambda_x\circ\lambda_y)(1)\\
                &=  \bigl((\phi\circ\lambda_x\circ\phi^{-1}) \circ
                        (\phi\circ\lambda_y\circ\phi^{-1})\bigr)(1)\\
                &=  (\lambda_{\phi(x)}\circ\lambda_{\phi(y)})(1)\\
                &=  \phi(x)\phi(y)\\
\end{aligned}
\end{equation*}
so that $\phi$ is an automorphism of the field $H$. Furthermore, $\phi(1)=1$
together with $\phi$ being $F$-semilinear implies $\phi(F)=F$.
\end{proof}
Take notice that $F$ is the centre of the quaternion skew field $H$ in case
\eqref{A} and so under these circumstances $\Aut(H/F)=\Aut(H)$.
\par
Suppose that $\phi\in \Aut(H/F)$ is $F$-linear or, equivalently, that $\phi$
fixes $F$ elementwise. Then $\phi\in\inner{H^*}$ is an inner automorphism of
the field $H$. In case \eqref{A}, this follows from the theorem of
Skolem-Noether \cite[Thm.~4.9]{jac-89}. In case \eqref{B}, any inner
automorphism of $H$ is trivial and $\phi=\id_H$, since any $h\in H^*\setminus
F^*$ is a double zero of the polynomial $h^2+t^2\in F[t]$, which is the minimal
polynomial of $h$ over $F$. So, by \eqref{eq:auto}, the group of all
\emph{$F$-linear automorphisms} of $\rip$ can be written in the form
\begin{equation}\label{eq:semidir2}
    \Autri\cap \GL(H_F) = \lambda(H^*) \rtimes \inner{H^*} .
\end{equation}
\par
Let $\parallel$ be any Clifford-like parallelism of $\dspH$. The group
appearing in \eqref{eq:semidir2} coincides with the group $\Autp\cap\GL(H_F)$
comprising all \emph{$F$-linear automorphisms} of $\parallel$ (see
\cite[Thm.~3.5]{havl+p+p-20a}). The problem to determine the full automorphism
group $\Autp$ without extra assumptions on $H$, $F$ or $\parallel$ seems to be
open. Partial solutions can be found \cite[Sect.~3]{havl+p+p-20a}. The examples
in \cite[Sect.~4]{havl+p+p-20a} show the existence of proper Clifford-like
parallelisms $\parallel$ satisfying $\Autp=\Autle=\Autri$ and also of proper
Clifford-like parallelisms $\parallel$ satisfying $\Autp\subset\Autle=\Autri$.

\subsection{Orbits under the group of inner automorphisms}

In this subsection $H$ denotes an $F$-algebra subject to \eqref{A}, that is, a
quaternion skew field with centre $F$. The following outcomes fail in case
\eqref{B}, since there the group of inner automorphisms is trivial.
\par
Recall that, given any $h\in H$, the \emph{trace} and the \emph{norm} of $h$
are the elements of $F$ defined, respectively, by $\tr(h)=h+\overline{h}$ and
$N(h)=h\overline{h}=\overline{h}h$, where $\overline{h}$ denotes the
\emph{conjugate} of $h$. The conjugation is an antiautomorphism of $H$ of order
$2$ that fixes $F$ elementwise. The identity $h^2-\tr(h)h+N(h)=0$ holds for any
$h\in H$. The norm $N$ is a multiplicative quadratic form and its associated
symmetric bilinear form is
\begin{equation}\label{eq:polar}
    \lire\colon H\x H\to F\colon
    (x,y)\mapsto \li x,y\re=\tr(x\ol{y})=x\ol{y}+y\ol{x} .
\end{equation}
The form $\lire$ is non-degenerate and so the mapping sending each subspace $X$
of $H_F$ to its orthogonal subspace $X^\perp$ is a polarity of $\bPH$.
\par
The next result is briefly mentioned in \cite[Rem.~4.5]{blunck+p+p-10a} and
\cite[p.~76, Ex.~10]{lam2} ($\Char F\neq 2$ only). For the sake of
completeness, a proof will be presented below.

\begin{lem}\label{lem:conjugates}
Given quaternions $q_1,q_2\in H$ there exists an inner automorphism of $H$
taking $q_1$ to $q_2$ if, and only if, $\tr(q_1)=\tr(q_2)$ and $N(q_1)=N(q_2)$.
\end{lem}

\begin{proof}
From $\tr(q_1)=\tr(q_2)$ and $N(q_1)=N(q_2)$, the quaternions $q_1,q_2$ are
zeros of the polynomial $m(t)=t^2-\tr(q_1)t+N(q_1)\in F[t]$. If $m(t)$ is
reducible over $F$, then $m(t)$ has no zeros in $H$ outside $F$. Thus $q_1\in
F$ and $m(t)=(t-q_1)^2$. Now $m(q_2)=0$ yields $q_2=q_1$, whence the identity
$\id_H$ is a solution. On the other hand, if $m(t)$ is irreducible over $F$,
then $\id_F$ can be extended in a unique way to an isomorphism $\gamma$ of the
commutative field $F1\oplus Fq_1\subset H$ onto the commutative field $F1\oplus
Fq_2\subset H$ such that $\gamma(q_1)=q_2$; see, for example,
\cite[Prop.~7.2.2]{cohn-03a}. By the theorem of Skolem-Noether
\cite[Thm.~4.9]{jac-89}, this $\gamma$ extends to an inner automorphism of $H$.
\par
The proof of the converse is straightforward.
\end{proof}

The above result describes the orbits under the action of the inner
automorphism group $\inner{H^*}$ on quaternions.\footnote{After extending $H$
to a \emph{projective line over $H$} by adding an extra point $\infty$, these
$\inner{H^*}$-orbits turn into orbits of the group of projectivities that fix
the points $0$, $1$ and $\infty$. This approach results in an alternative
description, as can be seen from \cite{havl-88b}.} By considering the vector
space $H_F$ as an affine space, the orbit of any $q\in H$ is the intersection
of the affine quadric $\{x\in H\mid N(x)=N(q)\}$ with the hyperplane $\{x\in
H\mid \tr(x)=\tr(q)\}$. Here, however, we aim at providing a description of the
orbits of the points of $\bPH$ under the action of $\inner{H^*}$. Since the
behaviour of the points of the plane $(F1)^\perp=\{x\in H\mid \tr(x)=0\}$ is
different from that of any other point, these points will be excluded in the
next proposition.

\begin{prop}\label{prop:orbit}
Let $H$ be a quaternion skew field with centre $F$ and let $Fq$, $q\in H^*$, be
a point of\/ $\bPH$ such that\/ $\tr(q)\neq 0$. Then the following hold.
\begin{enumerate}
\item\label{prop:orbit.a} The orbit of $Fq$ under the action of the group
    $\inner{H^*}$ of inner automorphisms of $H$ is a quadric of\/ $\bPH$,
    say $\cO_q$, which is given by the quadratic form
\begin{equation*}
    \omega_q\colon H\to F\colon x\mapsto \tr(q)^2\,N(x)- N(q)\,\tr(x)^2 .
\end{equation*}

\item\label{prop:orbit.b} If $q\in F^*$, then $\cO_q$ consists of a single
    point.

\item\label{prop:orbit.c} If $q\in H^*\setminus F^*$, then $\cO_q$ is an
    elliptic quadric, no line through $F1$ is tangent to $\cO_q$, and the
    polar form of $\omega_q$ is non-degenerate.
\end{enumerate}
\end{prop}

\begin{proof}
\eqref{prop:orbit.a} If $Fp$ is in the $\inner{H^*}$-orbit of $Fq$, then there
are $c\in F^*$ and $h\in H^*$ such that $p= c h^{-1}qh$. Consequently,
$N(p)=c^2 N(q)$ and $\tr(p)=c\tr(q)$. This entails
$\omega_q(p)=c^2\omega_q(q)=0$.
\par
Conversely, let a point $Fp'$, $p'\in H^*$, be given with $\omega_q(p')=0$.
Then $p'\neq 0$ implies $N(p')\neq 0$ and so $\tr(p')\neq 0$ follows from
$\omega_q(p')=0$. We define
\begin{equation*}
    p:=\tr(q)\tr(p')^{-1}p' .
\end{equation*}
Then $\tr (p)=\tr (q)\neq 0$, and $\omega_q(p)=0$ establishes $N(p)=N(q)$. Now
Lemma~\ref{lem:conjugates} implies the existence of an $h\in H^*$ such that
$p=h^{-1}qh$.
\par
\eqref{prop:orbit.b} The quadric $\cO_q$, $q\in F^*$, is the
$\inner{H^*}$-orbit of $F1$, whence it consists of this single point only.
\par
\eqref{prop:orbit.c} The point $F1$ is not in the $\inner{H^*}$-orbit of $Fq$
and so $F1$ is off the quadric $\cO_q$. From $q+\ol{q} = \tr(q)\in F^*$ and
$\omega_q(\ol{q})=0$, the line joining $Fq$ and $F1$ meets $\cO_q$ residually
at $F\ol{q}\neq Fq$ and so it is not tangent to $\cO_q$. Also, the point $Fq$
is a regular point of $\cO_q$. By the transitive action of the group
$\inner{H^*}$ on the points of $\cO_q$, the same applies to all other points of
$\cO_q$. The quadric $\cO_q$ cannot be ruled, because it does not contain any
point of the plane $\{x\in H\mid\tr(x)=0\}$.
\par
The polar form of $\omega_q$ is
\begin{equation*}
    (x,y)\mapsto \tr(q)^2\li x,y\re - 2 N(q)\tr(x)\tr(y)
            = \tr(q)^2\tr(x\ol{y}) -2N(q)\tr(x)\tr(y) .
\end{equation*}
If Char $F\neq 2$ then the polar form of $\omega_q$ is non-degenerate, since
otherwise $\cO_q$ would contain a singular point. In the case of $\Char F=2$
the form $\omega_q$ is non-degenerate, because it merely is a non-zero scalar
multiple of the non-degenerate alternating bilinear form $\lire$ from
\eqref{eq:polar}.
\end{proof}

\begin{prop}\label{prop:lineorbit}
Let $H$ be a quaternion skew field with centre $F$ and, in\/ $\bPH$, let $L$ be
a line that passes through the point $F1$ and is not contained in the plane\/
$(F1)^\perp$. Every plane through an arbitrary line in the $\inner{H^*}$-orbit
of $L$ contains infinitely many lines of this orbit.
\end{prop}
\begin{proof}
By virtue of the action of $\inner{H^*}$ on $\inner{H^*}(L)$, it is enough to
show the assertion for an arbitrary plane $E$ passing through $L$.
\par
On the line $L$, we can pick one point, say $Fq$, other than $F1$ such that
$\tr(q)\neq 0$. By Proposition~\ref{prop:orbit}, the orbit of $Fq$ is an
elliptic quadric $\cO_q$. Furthermore, the line $L$ is a bisecant of this
quadric that meets $\cO_q$ at $Fq$ and $F\ol{q}\neq Fq$. The plane $E$ contains
the bisecant $L$ of $\cO_q$ and so $E$ cannot be a tangent plane of $\cO_q$.
This implies that $E$ intersects $\cO_q$ along a regular conic. As $F$ is
infinite, so is this conic. By joining each of the points of the conic with
$F1$ we get infinitely many lines through $F1$ in the plane $E$. All of them
are in $\inner{H^*}(L)$.
\end{proof}

\begin{rem}\label{rem:infinite}
The orbit of any line $L\in\cAH$ under the group $\inner{H^*}$ is
infinite \cite[Thm.~3]{faith-58a}. This result was improved in \cite{waeh-81a}
by showing that any such orbit has cardinality $|F|$. Limited to the case of
quaternion skew fields and lines of $\cAH$ that are not in $(F1)^\perp$, the
last proposition enriches this result with a geometric insight.
\par
From \cite[Thm.~4.12]{havl+p+p-19a}, the group $\inner{H^*}$ acts transitively
on $\cAH$ if, and only if, $F$ is a formally real pythagorean field and $H$ is
an ``ordinary'' quaternion skew field with centre $F$. Precisely under these
circumstances, $\dspH$ admits no proper Clifford-like parallelisms.
\end{rem}

\subsection{Parallel classes fixed by automorphisms}

First, let $\dspH$ be a projective double space as specified in
Section~\ref{se:clifford}. Suppose that a left translation $\lambda_g$, $g\in
H^*$, acts as a non-identical collineation on $\bPH$. Hence $g\in H^*\setminus
F^*$. Any line $M\in\cLH$ is left parallel to its image $\lambda_g(M)$ and so
$\lambda_g$ fixes all left parallel classes. As we saw in
Lemma~\ref{lem:kernel}, $\riS(F1\oplus Fg)$ is the only right parallel class
that is fixed linewise under $\lambda_g$. If $\lambda_g$ fixes also all lines
of a left parallel class, then Lemma~\ref{lem:kernel} forces $\lambda_g$ to be
a right translation as well, that is, $g$ has to be in the centre of $H$. In
case \eqref{A} this gives a contradiction. In case \eqref{B}, $H$ is a
commutative field and so this condition imposes no restriction on $g$; due to
${\lep}={\rip}$, the given $\lambda_g$ fixes precisely one left parallel class
linewise, namely $\leS(F1\oplus Fg)$.
\par
For the rest of this subsection we confine ourselves to the case \eqref{A}.

\begin{prop}\label{prop:left-invariant}
Let $H$ be a quaternion skew field with centre $F$ and let $g\in H^*\setminus
F^*$. In\/ $\dspH$, a right parallel class is invariant under the left
translation $\lambda_g$ precisely when it is of the form $\riS(M)$, where $M$
is a line satisfying at least one of the following conditions:
\begin{gather}
    M=F1\oplus Fg ;                           \label{eq:M=1+g}  \\
    F1\subseteq M \subseteq g^{-1}(F1)^\perp .\label{eq:M-pencil}
\end{gather}
\end{prop}
\begin{proof}
(a) Suppose that \eqref{eq:M=1+g} holds. From
Lemma~\ref{lem:kernel}, all lines of the right parallel class $\riS(M)$ are
fixed under $\lambda_g$.
\par
(b) Suppose that a line $M$ satisfies \eqref{eq:M-pencil}. The line $M^\perp$
is left parallel and right parallel to $M$ (see \cite[Cor.~4.4]{havl+p+p-19a})
and it is contained in $(F1)^\perp$. The line $gM$ is also left parallel to
$M$. As $M^\perp$ and $gM$ are incident with the plane $(F1)^\perp$, they share
a common point and so they must coincide. Taking into account that
$\lambda_g\in\Autri$ and $M^\perp\rip M$ we obtain
$\lambda_g\bigl(\riS(M)\bigr)=\riS(gM)=\riS(M^\perp)=\riS(M)$, as required.
\par
(c) Conversely, any $\lambda_g$-invariant right parallel class can be written
as $\riS(M)$ with $F1\subseteq M$. Then $\lambda_g(M)\lep M \rip \lambda_g(M)$.
Again from \cite[Cor.~4.4]{havl+p+p-19a}, there are only two possibilities.
First, $\lambda_g(M)=gM=M$, which implies $Fg\subseteq M$ and establishes
\eqref{eq:M=1+g}. Second, $\lambda_g(M)=gM=M^\perp$. From $F1\subseteq M$ we
obtain $\lambda_g(M) = M^\perp \subseteq (F1)^\perp$. Applying $\lambda_g^{-1}$
results in $M\subseteq g^{-1}(F1)^\perp$, whence \eqref{eq:M-pencil} holds.
\end{proof}

\begin{rem}\label{rem:cases}
 Figures \ref{fig:1} and \ref{fig:2}
depict the possible cases in Proposition~\ref{prop:left-invariant} under the
assumption $\Char F\neq 2$ and $\Char F = 2$, respectively. In all cases, there
are distinct points $F1$ and $Fg$ as well as distinct planes $(F1)^\perp$ and
$g^{-1}(F1)^\perp$. Furthermore, $(F1)^\perp \cap \bigl(g^{-1}(F1)^\perp\bigr)
= (F1\oplus Fg)^\perp$.
\par
The pictures on the left-hand side show the situation when $F1\not\subseteq
g^{-1}(F1)^\perp$ or, in other words, when $Fg\not\subseteq (F1)^\perp$, which
in turn is equivalent to $\tr(g)\neq 0$. Here there are no lines $M$ subject to
\eqref{eq:M-pencil}. The pictures on the right-hand side show the opposite
situation. Here the set of all lines $M$ that satisfy \eqref{eq:M-pencil}
comprises a pencil of lines. In detail, the circumstances are as follows.
\begin{figure}[!ht]\unitlength1cm 
  \centering
  \begin{picture}(10,3.5)
  \small
    \put(0,0.0){\includegraphics[height=8\unitlength]{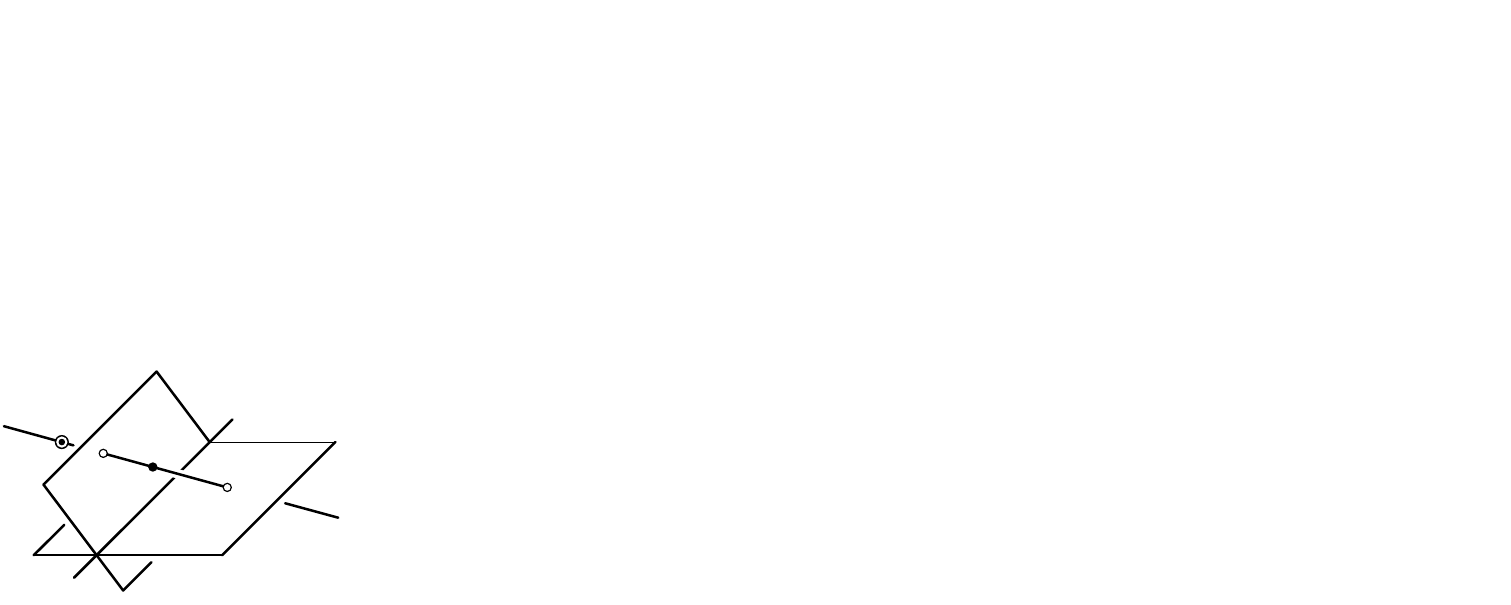}}
    \put(0.55,2.3){$F1$}
    \put(1.85,2.){$Fg$}
    \put(1.05,1.6){$Fg'$}
    \put(2.65,1.75){$F(g-\ol{g})$}
    \put(2.25,3.0){$g^{-1}(F1)^\perp$}
    \put(2.6,0.3){$(F1)^\perp$}
    \put(-0.5,0.05){$(F1\oplus Fg)^\perp$}
    \put(5.5,0.0){\includegraphics[height=8\unitlength]{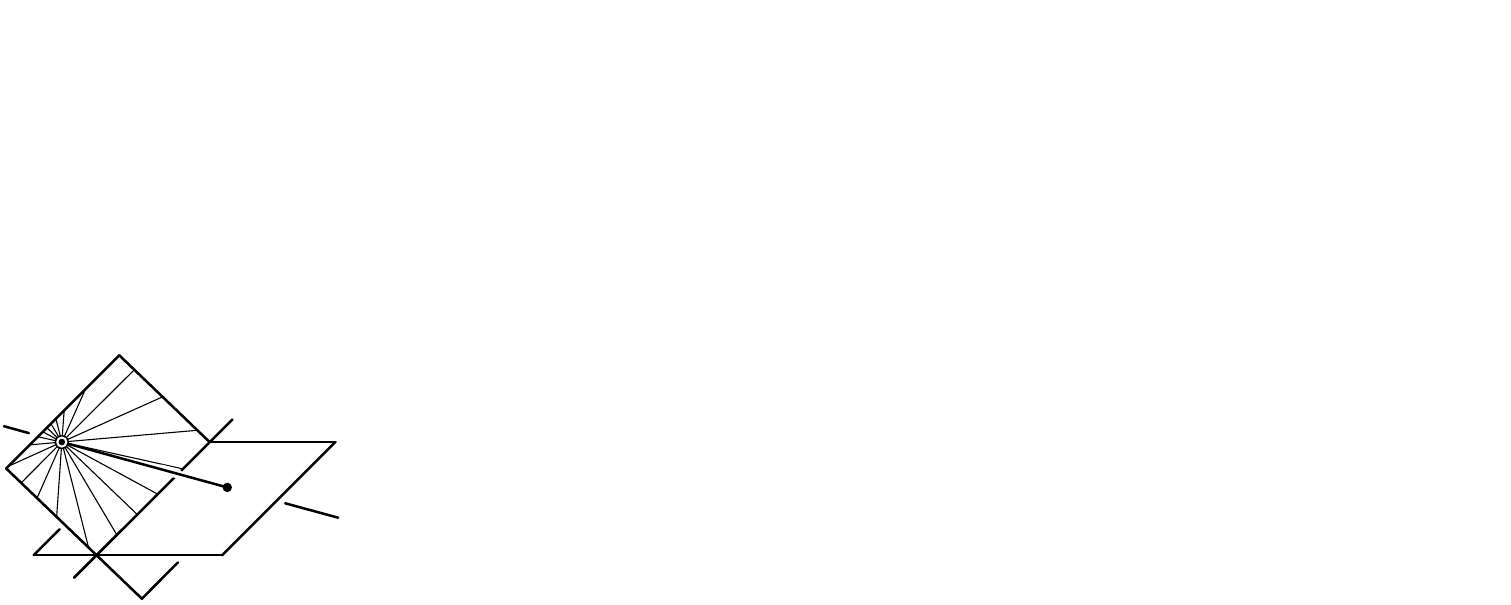}}
    \color{white}\linethickness{0.4\unitlength}
    \put(6.5,2.33){\line(1,0){0.5}}
    \color{black}
    \put(6.52,2.22){$F1$}
    \put(8.3,1.75){$Fg$}
    \put(7.25,3.25){$g^{-1}(F1)^\perp$}
    \put(8.1,0.3){$(F1)^\perp$}
    \put(5.0,0.05){$(F1\oplus Fg)^\perp$}
  \end{picture}
  \caption{$\Char F\neq 2$}\label{fig:1}
\end{figure}
\par
Figure \ref{fig:1}, left: the line $F1\oplus Fg$ intersects the plane
$(F1)^\perp$ at $F(g-\ol{g})$ and the plane $g^{-1}(F1)^\perp$ at $Fg'$,
$g':=g^{-1}(g-\ol{g})$; the points $F1$, $Fg$, $F(g-\ol{g})$ and $Fg'$ are
mutually distinct; the lines $F1\oplus Fg$ and $(F1\oplus Fg)^\perp$ are skew.
\par
Figure \ref{fig:1}, right: $(F1\oplus Fg)\cap (F1)^\perp= Fg$, $(F1\oplus
Fg)\cap {g^{-1}(F1)^\perp} = F1$; the lines $F1\oplus Fg$ and $(F1\oplus
Fg)^\perp$ are skew.
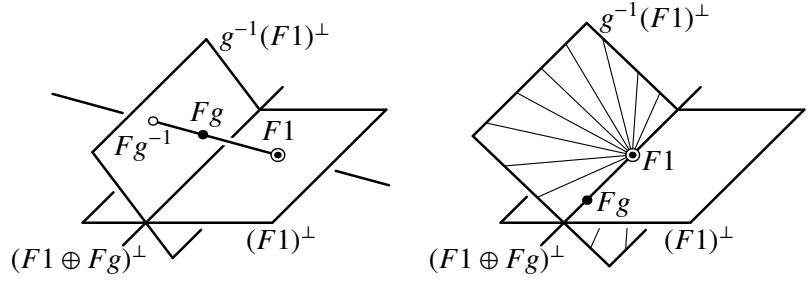
\begin{figure}[!ht]\unitlength1cm 
  \centering
  \begin{picture}(10,3.5)
  \small
    \put(0,0.0){\includegraphics[height=8\unitlength]{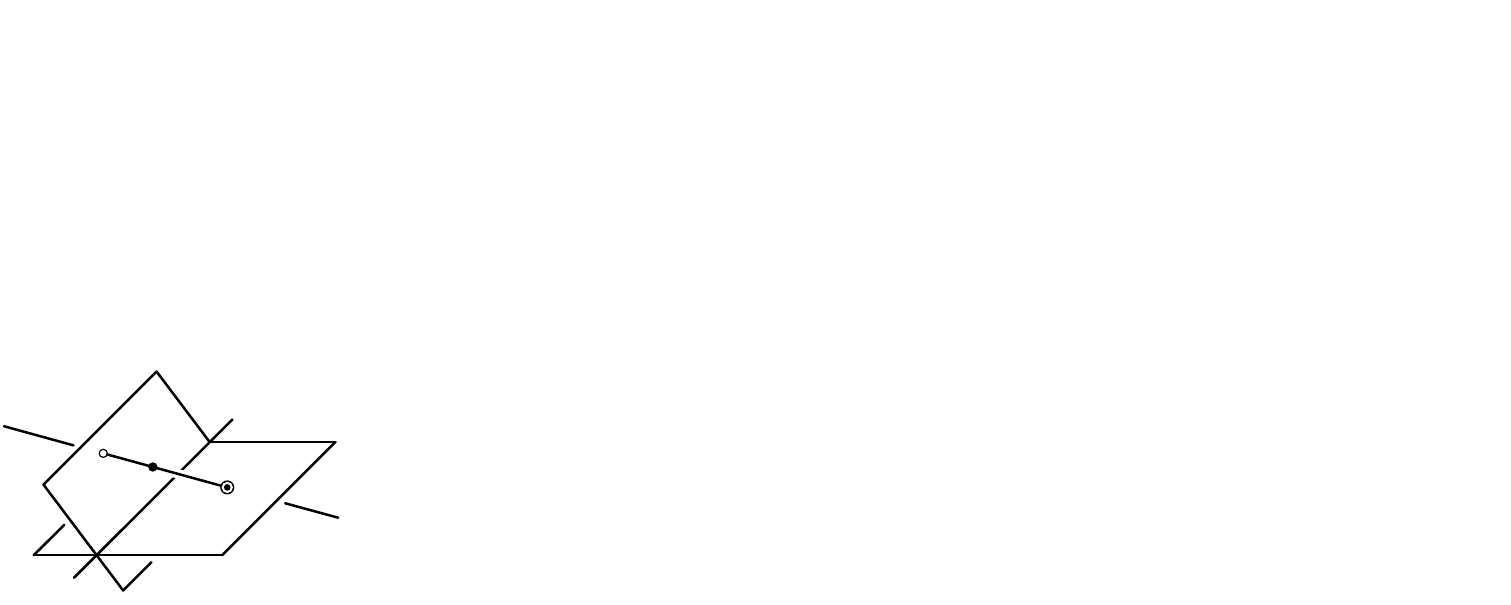}}
    \put(2.8,1.75){$F1$}
    \put(1.85,2.){$Fg$}
    \put(0.85,1.55){$Fg^{-1}$}
    \put(2.25,3.0){$g^{-1}(F1)^\perp$}
    \put(2.6,0.3){$(F1)^\perp$}
    \put(-0.5,0.05){$(F1\oplus Fg)^\perp$}
    \put(5.5,0.0){\includegraphics[height=8\unitlength]{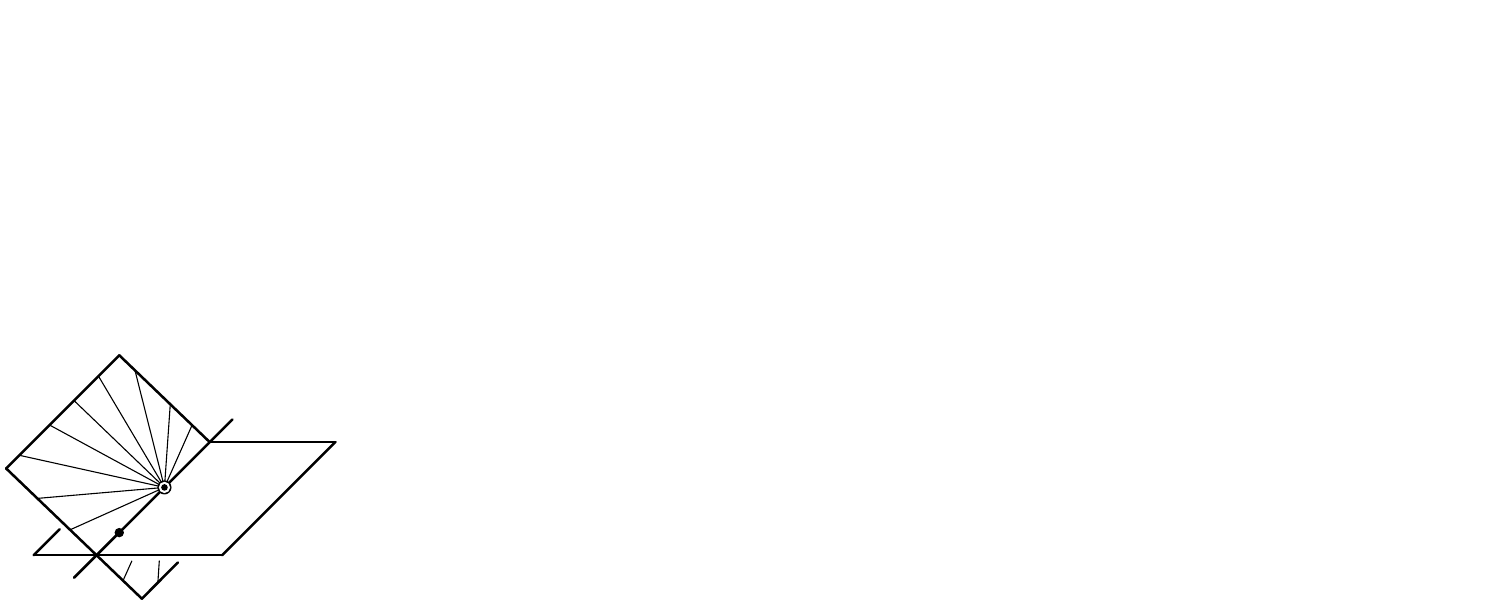}}
    \color{black}
    \put(7.8,1.35){$F1$}
    \put(7.2,0.8){$Fg$}
    \put(7.25,3.25){$g^{-1}(F1)^\perp$}
    \put(8.1,0.3){$(F1)^\perp$}
    \put(5.0,0.05){$(F1\oplus Fg)^\perp$}
  \end{picture}
  \caption{$\Char F= 2$}\label{fig:2}
\end{figure}
\par
Figure \ref{fig:2}, left: $(F1\oplus Fg) \cap (F1)^\perp = F1$, $(F1\oplus Fg)
\cap g^{-1}(F1)^\perp = Fg^{-1}$; the points $F1$, $Fg$ and $Fg^{-1}$ are
mutually distinct; the lines $F1\oplus Fg$ and $(F1\oplus Fg)^\perp$ are skew.
\par
Figure \ref{fig:2}, right: the line $F1\oplus Fg$ coincides with $(F1\oplus
Fg)^\perp$.
\par
Finally, note that the situations depicted on the right-hand side, namely
$Fg\subseteq (F1)^\perp$, comprises precisely the cases when the left
translation $\lambda_g$ acts as an involution on the projective space.
\end{rem}

\section{Main results}\label{se:main}

The definition of a Clifford-like parallelism in \cite[Def.~3.2]{havl+p+p-19a}
is essentially based on a given projective double space $\dsp$. We are thus led
to the problem of whether or not distinct projective double spaces can share a
Clifford-like parallelism.
\begin{thm}\label{thm:double-double}
Let\/ $\dspH$ be a projective double space, where $H$ is an $F$-algebra subject
to\/ \eqref{A} or\/ \eqref{B}. Furthermore, let\/ $\lepp$ and\/ $\ripp$ be
parallelisms such that\/ $\bigl(\bPH,\lepp,\ripp\bigr)$ is also a projective
double space. Suppose that a parallelism\/ $\parallel$ of\/ $\bPH$ is
Clifford-like with respect to both double space structures. Then, possibly up
to a change of the attributes ``left'' and ``right'' in one of these double
spaces, ${\lep}={\lepp}$ and\/ ${\rip}={\ripp}$.
\end{thm}
\begin{proof}
First, we consider case \eqref{A}. We take any line of the star $\cAH$. We
noted in Remark~\ref{rem:infinite} that the orbit of this line under the group
$\inner{H^*}$ of all inner automorphisms of $H$ is infinite. Thus there are
three mutually distinct lines, say $L_1$, $L_2$ and $L_3$, in this orbit. From
\cite[Thm.~4.10]{havl+p+p-19a}, the $\parallel$-classes of these lines are of
the same kind w.r.t.\ $\dspH$, \emph{i.e.}, we have either
$\cS(L_{n})=\leS(L_{n})$ for all $n\in\{1,2,3\}$ or $\cS(L_{n})=\riS(L_{n})$
for all $n\in\{1,2,3\}$.
\par
Next, we turn to case \eqref{B}. There exist three mutually distinct
lines $L_1, L_2,L_3\in\cAH$. Their $\parallel$-classes are of the same kind
w.r.t.\ $\dspH$ due to ${\lep}={\rip}={\parallel}$.
\par
In both cases, the parallel classes $\cS(L_n)$, $n\in\{1,2,3\}$,
are mutually distinct. Consequently, among them there are at least two distinct
classes of the same kind w.r.t.\ the double space $(\bP(H_F),\lepp,\ripp)$. Up
to a change of notation, we may assume $\cS(L_n)=\riS(L_n)=\riS'(L_n)$ for
$n\in\{1,2\}$. Now Proposition~\ref{prop:two} shows that the Clifford
parallelisms $\rip$ and $\ripp$ coincide. This in turn forces ${\lep} =
{\lepp}$, since the left parallelism is uniquely determined by the right one
(see \cite[pp.~75--76]{karz+k-88} or \cite[\S6]{herz-77a}).
\end{proof}

\begin{cor}\label{cor:proper}
Any Clifford-like parallelism\/ $\parallel$ of\/ $\dspH$ other than\/ $\lep$
and\/ $\rip$ is not Clifford.
\end{cor}
\begin{proof}
Assume to the contrary that ${\parallel}=:{\lepp}$ is Clifford. Then there is a
parallelism, say $\ripp$, such that $\bigl(\bPH,\lepp,\ripp\bigr)$ is a
projective double space. Applying Theorem~\ref{thm:double-double} gives
therefore ${\parallel}={\lep}$ or ${\parallel}={\rip}$, a contradiction.
\end{proof}

The above corollary, when restricted to case \eqref{A}, is just a reformulation
of \cite[Thm.~4.15]{havl+p+p-19a}. Therefore, the rather technical proof in
\cite{havl+p+p-19a}, which relies on $H$ being a quaternion skew field, can now
be avoided.
\par
Our final results provide the announced characterisations of Clifford
parallelisms among Clifford-like parallelisms.

\begin{thm}\label{thm:main}
Let\/ $\parallel$ be a Clifford-like parallelism of\/ $\dspH$, where $H$ is an
$F$-algebra subject to\/ \eqref{A} or\/ \eqref{B}. Then the following
assertions are equivalent.
\begin{enumerate}
\item\label{thm:main.a} The parallelism\/ $\parallel$ is Clifford.

\item\label{thm:main.b} The parallelism\/ $\parallel$ admits an
    automorphism $\beta\in \Autp$ that stabilises all its parallel classes
    and acts as a non-identical collineation on the projective space\/
    $\bPH$.
\end{enumerate}
\end{thm}
\begin{proof}
\eqref{thm:main.a} $\Rightarrow$ \eqref{thm:main.b}. There exists a $g\in
H^*\setminus F^*$.
Corollary~\ref{cor:proper} shows that ${\parallel} = {\lep}$ or
${\parallel} = {\rip}$. In the first case the left translation $\lambda_g$ has
the required properties, in the second case the same applies to the right
translation $\rho_g$.
\par
\eqref{thm:main.b} $\Rightarrow$ \eqref{thm:main.a} In case \eqref{B},
${\lep}={\rip}$ implies that ${\parallel}={\lep}$ is Clifford.
\par
From now on we deal with case \eqref{A} only. We select one line $N_1$ through
$F1$ that is not in $(F1)^\perp$. We assume w.l.o.g.\ that the parallel class
$\cS(N_1)$ is a \emph{left} parallel class. (Otherwise, we have to interchange
the attributes ``left'' and ``right'' in what follows.) Let $g:=\beta(1)$. We
consider the left translation $\lambda_g$ and the product
\begin{equation}\label{eq:alpha}
    \alpha:=\lambda_g^{-1}\circ\beta .
\end{equation}
\par
We choose one $N\in\inner{H^*}(N_1)$. Then the parallel class $\cS(N)$ is a
left parallel class. Thus
\begin{equation}\label{eq:N}
        N \lep \beta(N)\lep g^{-1}\beta(N) =\alpha(N).
\end{equation}
Formula \eqref{eq:N} and $\alpha(1)=1\in N$ together force $\alpha(N)=N$. By
Proposition~\ref{prop:lineorbit}, every plane through $N$ contains at least two
lines from the orbit $\inner{H^*}(N_1)$, and so any such plane is fixed under
$\alpha$. The lines and planes through $F1$ are the ``points'' and ``lines'' of
a projective plane; ``incidence'' is given by symmetrised inclusion. Our
$\alpha$ acts on this projective plane as a collineation. By the above, all
``lines'' through the ``point'' $N$ are fixed under $\alpha$, and so $N$ serves
as a ``centre'' of this collineation. But $N$ may vary in the orbit
$\inner{H^*}(N_1)$, which comprises more than one line by the theorem of
Cartan-Brauer-Hua \cite[(13.17)]{lam-01a}. Consequently, this collineation has
more than one ``centre'', that is, $\alpha$ fixes all lines of the star $\cAH$.
\par
We now consider the action of $\alpha$ on the projective space $\bPH$. Since
all lines of $\cAH$ are fixed, $\alpha$ acts as a perspective collineation with
centre $F1$. This implies that $\alpha$ is $F$-linear. Since $\alpha$ and
$\lambda_g^{-1}$ are $F$-linear, so is $\beta$. From $\beta\in\Autp\cap\GL(H_F)
= \Autle\cap\GL(H_F)$ (see \cite[Thm.~3.5]{havl+p+p-20a}) and
$\lambda_g^{-1}\in\Autle\cap\GL(H_F)$ follows $\alpha\in\Autle\cap\GL(H_F)$.
Now pick any line $L\in\cLH$. The left parallel line to $L$ through $F1$ is
fixed under $\alpha\in\Autle$, whence we have $L\lep\alpha(L)$. On the other
hand, $L$ is incident with at least one plane through $F1$. This plane is
$\alpha$-invariant. Therefore the left parallel lines $L$ and $\alpha(L)$ are
coplanar, which in turn implies $L=\alpha(L)$. So we arrive at $\alpha=c\id_H$
for some $c\in F^*$. Now, using $\alpha(1)=1$, we end up with $\alpha=\id_H$.
\par
Next, we give an explicit description of $\beta$. By virtue of
\eqref{eq:alpha}, our assumption that $\beta$ does not fix all lines of $\bPH$,
and $\alpha=\id_H$, we have
\begin{equation*}
    \beta =\lambda_g \mbox{~~and~~} g\in H^*\setminus F^*.
\end{equation*}
\par
Finally, we claim that ${\parallel}={\lep}$. Assume to the contrary that
${\parallel}\neq{\lep}$. So there is a line $M_1$ with $\cS(M_1)=\riS(M_1)$ and
$F1\subseteq M_1$. Then $\cS(M)=\riS(M)$ for all lines $M\in \inner{H^*}(M_1)$,
which forces
\begin{equation}\label{eq:cM_1}
    \beta\bigl(\riS(M)\bigr)=\riS(M) \mbox{~~for all~~} M\in \inner{H^*}(M_1) .
\end{equation}
We now distinguish three cases.
\par
\emph{Case}~(i). Let $F1\not\subseteq g^{-1}(F1)^\perp$. From
Proposition~\ref{prop:left-invariant} and \eqref{eq:cM_1}, any line
$M\in\inner{H^*}(M_1)$ has to satisfy \eqref{eq:M=1+g}. This implies
$\inner{H^*}(M_1) = \{F1\oplus Fg\}$ and contradicts the theorem of
Cartan-Brauer-Hua \cite[(13.17)]{lam-01a}, which says $|\inner{H^*}(M_1)|>1$.
\par
\emph{Case}~(ii). Let $F1\subseteq g^{-1}(F1)^\perp$ and $M_1\not\subseteq
(F1)^\perp$. We choose any plane $E$ other than $g^{-1}(F1)^\perp$ through the
line $M_1$. Let $\cM_{E}$ denote the set of all lines that are incident with
$E$ and belong to $\inner{H^*}(M_1)$. By Proposition~\ref{prop:lineorbit}, the
set $\cM_{E}$ is infinite. From Proposition~\ref{prop:left-invariant} and
\eqref{eq:cM_1}, any line $M\in\cM_{E}$ has to satisfy \eqref{eq:M=1+g} or
\eqref{eq:M-pencil}, that is $M=F1\oplus Fg$ or $M=g^{-1}(F1)^\perp\cap E$.
This implies $|\cM_{E}|\leq 2$, an absurdity.
\par
\emph{Case}~(iii). Let $F1\subseteq g^{-1}(F1)^\perp$ and $M_1\subseteq
(F1)^\perp$. From Remark~\ref{rem:cases}, this applies precisely when
\begin{equation}\label{eq:M1}
    M_1=F1\oplus Fg=(F1)^\perp\cap\bigl(g^{-1}(F1)^\perp\bigr) ;
\end{equation}
see the right-hand side of Figure~\ref{fig:2}. The plane $(F1)^\perp$ is
$\inner{H^*}$-invariant, whence it contains all lines of $\inner{H^*}(M_1)$.
From Proposition~\ref{prop:left-invariant} and \eqref{eq:cM_1}, any line
$M\in\inner{H^*}(M_1)$ has to satisfy \eqref{eq:M=1+g} or \eqref{eq:M-pencil}.
By virtue of the second equation in \eqref{eq:M1}, this implies
$\inner{H^*}(M_1) = \{F1\oplus Fg\}$ and, as in Case~(i), contradicts the
theorem of Cartan-Brauer-Hua.
\end{proof}

\begin{rem}
Note that, as a consequence of the previous theorem, the group of automorphisms
that preserve all parallel classes with respect to a given Clifford-like
parallelism $\parallel$ of $\dspH$ is contained in $\GL(H_F)$. Moreover this
group is the group of left translations (or right translations respectively)
precisely when ${\parallel}={\rip}$ (respectively ${\parallel}={\lep}$). If, on
the other hand, $\parallel$ is a proper Clifford-like parallelism, then this
group is the group of all $\lambda_g$ with $g\in F^*$, thus, from the
projective point of view, it comprises only the identity map.
\end{rem}

\begin{thm}\label{thm:new1}
Let\/ $\parallel$ be a Clifford-like parallelism of\/ $\dspH$, where $H$ is an
$F$-algebra subject to\/ \eqref{A} or\/ \eqref{B}. Then the following
assertions are equivalent.
\begin{enumerate}
\item\label{thm:new1.a} The parallelism\/ $\parallel$ is Clifford and\/
    ${\lep}\neq {\rip}$.

\item\label{thm:new1.b} The parallelism\/ $\parallel$ admits an
    automorphism $\beta\in \Autp$ that stabilises a single parallel class
    of\/ $\parallel$ and, furthermore, fixes all lines of this particular
    parallel class.
\end{enumerate}
\end{thm}
\begin{proof}
\eqref{thm:new1.a} $\Rightarrow$ \eqref{thm:new1.b}. Corollary~\ref{cor:proper}
shows that ${\parallel} = {\lep}$ or ${\parallel} = {\rip}$. Let, for example,
${\parallel}={\rip}$. We infer from ${\lep}\neq {\rip}$ that $H$ is a
quaternion skew field. There exists a $g\in H\setminus \bigl(F1 \cup
(F1)^\perp\bigr)$; cf.\ the left-hand sides of Fig.~\ref{fig:1} and
Fig.~\ref{fig:2} for illustrations. Then no line $M\in\cLH$ satisfies
\eqref{eq:M-pencil}. By Proposition \ref{prop:left-invariant},
$\beta:=\lambda_g$ stabilises a single right parallel class, namely
$\riS(F1\oplus Fg)$, and, furthermore, $\beta$ fixes all lines of
$\riS(F1\oplus Fg)$.
\par
\eqref{thm:new1.b} $\Rightarrow$ \eqref{thm:new1.a}. The only $\beta$-invariant
parallel class can be written in the form $\cS(L)$ with $L\in\cAH$. Let us
assume that $\cS(L)$ is a right parallel class. Since all lines of $\riS(L)$
are fixed under $\beta$, we obtain
$\beta\in\Kernel\bigl(H,\riS(L)\bigr)^*=\lambda(L^*)$ from
Lemma~\ref{lem:kernel}. Consequently, all left parallel classes are stabilised
under $\beta$, whence none of them is a parallel class of $\parallel$. This
shows ${\lep}\neq{\parallel}={\rip}$.
\end{proof}

\begin{thm}\label{thm:new2}
Let\/ $\parallel$ be a Clifford-like parallelism of\/ $\dspH$, where $H$ is an
$F$-algebra subject to\/ \eqref{A} or\/ \eqref{B}. Then the following
assertions are equivalent.
\begin{enumerate}
\item\label{thm:new2.a} The parallelism\/ $\parallel$ is Clifford and\/
    ${\lep}={\rip}$.
\item\label{thm:new2.b} If an automorphism $\beta\in\Autp$\/ fixes all
    lines of at least one parallel class of\/ $\parallel$, then all
    parallel classes of\/ $\parallel$ are stabilised under $\beta$ .
\item\label{thm:new2.c} The parallelism\/ $\parallel$ admits an
    automorphism $\beta\in \Autp$ that stabilises all its parallel classes,
    fixes at least one of its parallel classes linewise, and acts as a
    non-identical collineation on the projective space \/$\bPH$.
\end{enumerate}
\end{thm}
\begin{proof}
\eqref{thm:new2.a} $\Rightarrow$ \eqref{thm:new2.b}. We have
${\parallel}={\lep}={\rip}$. Let $\beta\in\Autp$ fix all lines of a
right\footnote{We use the attributes ``left'' and ``right'' in accordance with
the general situation, as described elsewhere. Of course, the distinction
between ``left'' and ``right'' is immaterial here.} parallel class, which will
be written as $\riS(L)$ with $L\in\cAH$. From Lemma~\ref{lem:kernel},
$\beta\in\Kernel\bigl(H,\riS(L)\bigr)^*=\lambda(L^*)$ and so $\beta$ stabilises
all left parallel classes or, said differently, all $\parallel$-classes.
\par
\eqref{thm:new2.b} $\Rightarrow$ \eqref{thm:new2.c}. We may assume w.l.o.g.\
that there exists a line $L\in\cAH$ with the property $\cS(L)=\riS(L)$. There
is a $g\in L^*\setminus F^*$. The left translation $\lambda_g=:\beta$ fixes all
lines of $\cS(L)=\riS(L)$ and acts as a non-identical collineation on $\bPH$.
So, by our assumption, $\beta$ stabilises all $\parallel$-classes. Thus $\beta$
meets all the requirements appearing in \eqref{thm:new2.c}.
\par
\eqref{thm:new2.c} $\Rightarrow$ \eqref{thm:new2.a}. We may assume w.l.o.g.\
that $\beta$ fixes all lines of a right parallel class, $\riS(L)=\cS(L)$ with
$L\in\cAH$. There are two possibilities.
\par
\emph{Case}~(i). ${\lep}\neq{\rip}$. Theorem \ref{thm:main} gives that
$\parallel$ is a Clifford parallelism of $\dspH$. From
\cite[Cor.~4.3]{havl+p+p-19a}, the parallelisms $\lep$ and $\rip$ have no
parallel classes in common. Consequently, $\riS(L)$ being one of the
$\parallel$-classes yields ${\parallel}={\rip}$. From Lemma~\ref{lem:kernel},
the given automorphism $\beta$ is a left translation $\lambda_g$ for some $g\in
H^*$. Since $\beta$ acts non-identical on $\bPH$, we have $g\in H^*\setminus
F^*$. Hence, by Proposition~\ref{prop:left-invariant}, at least one right
parallel class is not stabilised under $\beta$, a contradiction.
\par
\emph{Case}~ (ii). ${\lep}={\rip}$. Now ${\lep}={\rip}$ is the only
Clifford-like parallelism of the projective double space $\dspH$, whence
$\parallel$ turns out to be Clifford.
\end{proof}

\bibliographystyle{amsplain}
\providecommand{\bysame}{\leavevmode\hbox to3em{\hrulefill}\thinspace}
\providecommand{\MR}{\relax\ifhmode\unskip\space\fi MR }
\providecommand{\MRhref}[2]{%
  \href{http://www.ams.org/mathscinet-getitem?mr=#1}{#2}
} \providecommand{\href}[2]{#2}

\noindent
Hans Havlicek\\
Institut f\"{u}r Diskrete Mathematik und Geometrie\\
Technische Universit\"{a}t\\
Wiedner Hauptstra{\ss}e 8--10/104\\
A-1040 Wien\\
Austria\\
\texttt{havlicek@geometrie.tuwien.ac.at}
\par~\par
\noindent Stefano Pasotti\\
DICATAM-Sez.\ Matematica\\
Universit\`{a} degli Studi di Brescia\\
via Branze, 43\\
I-25123 Brescia\\
Italy\\
\texttt{stefano.pasotti@unibs.it}
\par~\par
\noindent
Silvia Pianta\\
Dipartimento di Matematica e Fisica\\
Universit\`{a} Cattolica del Sacro Cuore\\
via Trieste, 17\\
I-25121 Brescia\\
Italy\\
\texttt{silvia.pianta@unicatt.it}
\end{document}